\documentclass[11pt,letterpaper]{article}

\usepackage[T1]{fontenc}

\usepackage[latin1]{inputenc}							
\usepackage{amsmath}									

\usepackage{xcolor}
\definecolor{bl}{rgb}{0.0,0.2,0.6}

\usepackage{sectsty}
\usepackage[compact]{titlesec} 
\allsectionsfont{\color{bl}\normalsize\scshape\selectfont}

\makeatletter							
\def\printtitle{
    {\color{bl} \centering \Large \sc \textbf{\@title}\par}}		
\makeatother							

\title{The Newtonian potential and \\the demagnetizing factors of the general ellipsoid}

\makeatletter							
\def\printauthor{
    {\centering \small  \@author}}				
\makeatother							

\author{%
\vspace{20pt}
	{\sc Giovanni Di Fratta} \\
	\vspace{6pt}
	{\footnotesize CMAP, \'Ecole Polytechnique,\\ route de Saclay,\\ 91128 Palaiseau Cedex,\\ FRANCE} \\
	\vspace{20pt}
	}

\usepackage{fancyhdr}
\usepackage[runin]{abstract}			
\abslabeldelim{}						%
\usepackage{amsmath,amssymb,latexsym,amsthm}
\usepackage{bbold}
\normalfont
      plus .94\fontdimen3\font
     minus .94\fontdimen4\font
		
\newcommand{\assign}{:=}
\newcommand{\mathd}{\mathrm{d}}
\newcommand{\nobracket}{}
\newcommand{\tmem}[1]{{\em #1\/}}
\newcommand{\tmmathbf}[1]{\ensuremath{\boldsymbol{#1}}}
\newcommand{\tmname}[1]{\textsc{#1}}
\newcommand{\tmop}[1]{\ensuremath{\operatorname{#1}}}

\newcommand{\tmtextit}[1]{{\itshape{#1}}}
\newcommand{\tmtextsl}[1]{{\slshape{#1}}}

\newcommand{\RR}{\mathbb{R}}
\newcommand{\NN}{\mathbb{N}}

\newtheorem{theorem}{\bf Theorem}[section]

\newtheorem{proposition}{\bf Proposition}[section]

\newtheorem{corollary}{\bf Corollary}[section]
\theoremstyle{remark}
\newtheorem{remark}{\bf Remark}[section]

\begin{document}
\printtitle 

\printauthor

\begin{abstract}
 The objective of this paper is to present a modern and concise new
  derivation for the explicit expression of the interior and exterior
  Newtonian potential generated by homogeneous ellipsoidal domains in $\RR^N$
  (with $N \geqslant 3$). The very short argument is essentially based on the
  application of {\tmname{Reynolds}} transport theorem in connection with
  {\tmname{Green-Stokes}} integral representation formula for smooth functions
  on bounded domains of $\RR^N$, which permits to reduce the $N$-dimensional
  problem to a 1-dimensional one. Due to its physical relevance, a separate
  section is devoted to the the derivation of the {\tmem{demagnetizing
  factors}} of the general ellipsoid which are one of the most fundamental
  quantities in ferromagnetism.
\end{abstract}
\vspace{20pt}

\section{Historical Introduction, Motivations}
The computation of the gravitational potential induced by an homogeneous
ellipsoid was one of the most important problems in mathematics for more than
two centuries after {\tmname{Newton}} enunciated his universal law of
gravitation
{\cite{beltrami1880sulla,daniele1911sul,daniele1911sull,newton1687philosophiae,rahman2001newtonian,shahgholian1991newtonian}}.

Once fixed an ellipsoidal domain $\Omega$ of $\RR^3$, the problem consists in
finding an {\tmem{explicit}} expression for the Newtonian potential induced by
a constant mass/charge density on $\Omega$ {\cite{dibenedetto2010partial}}:
\begin{equation}
  \mathcal{N}_{\Omega} [1_{\Omega}] (x) {\assign} 
  \frac{1}{4 \pi} \int_{\Omega} \frac{1}{| x - y |^{}} \mathd y,
  \label{eq:newtonianpotential}
\end{equation}
in the internal points of $\Omega$ ({\tmem{interior problem}}) and in the
exterior points of $\Omega$ ({\tmem{exterior problem}}).

In the case of a homogeneous {\tmem{spherically symmetric}} region,
{\tmname{Newton}} (in 1687) proved what nowadays is known as
{\tmname{Newton}}'s {\tmname{shell theorem}}
{\cite{arens1990newton,newton1687philosophiae}}: {\tmem{if $\Omega$ is an
homogeneous\footnote{By an {\tmem{homogeneous}} region of $\RR^N$ we mean an
open and connected subset of $\RR^N$ endowed by a constant density of
masses/charges.} spherical region centered at the origin, then for all $t >
1$, $\mathcal{N}_{t \Omega \backslash \Omega} [1_{t \Omega \backslash
\Omega}]$ is constant in $\Omega$, i.e.}} $t \Omega \backslash \Omega$ (the so
called {\tmem{hollow ball}}) {\tmem{induces no gravitational force inside}}
$\Omega$. {\tmem{Moreover, a spherically symmetric body affects external
objects gravitationally as though all of its mass were concentrated at a point
at its center.}}

For what concerns an ellipsoidal domain $\Omega$ of $\RR^3$, the
{\tmem{ellipsoidal}} {\tmem{interior}} problem was for the first time solved
by {\tmname{Gauss}} (in 1813) by the means of what it is in present days known
as the {\tmname{Gauss}} divergence theorem {\cite{gauss1813theoria}}. Later,
in 1839, {\tmname{Dirichlet}} proposed a solution of the interior problem
based on the theory of {\tmname{Fourier}} integrals {\cite{Kroneker1969}}. 

The
results of {\tmname{Gauss}} and {\tmname{Dirichlet}} can be summarized by
saying that if $\Omega$ is an ellipsoidal domain centered at the origin, then
the gravitational potential induced by an homogeneous ellipsoid in its
internal points is a second order polynomial. In other terms:
\begin{equation}
  \mathcal{N}_{\Omega} [1_{\Omega}] (x) = c - P_{} x \cdot x \quad \forall x
  \in \Omega, \label{eq:quadratic}
\end{equation}
for some constant $c \in \RR$ and some matrix $P \in \RR^{3 \times 3}$ whose
values can be expressed in terms of elliptic integrals
{\cite{KelloggB2010,stratton2007electromagnetic}}.

\begin{remark}
  For the sake of completeness we recall that the converse statement (the {\tmem{inverse
homogeneous ellipsoid problem}}) is also true
{\cite{dibenedetto1986bubble,dibenedetto2010partial,karp1994newtonian}},
namely: \tmtextsl{if $\Omega$ is a bounded domain of $\RR^N$ such that $\RR^N
\backslash \Omega$ is connected and} (\ref{eq:quadratic}) \tmtextsl{holds,
then $\Omega$ is an ellipsoid}. Historically speaking the inverse homogeneous
ellipsoid problem was for the first time solved by {\tmname{Dive}}
{\cite{dive1931attraction}} in 1931 for $N = 3$ and in 1932 by
{\tmname{H{\"o}lder}} {\cite{holder1932potentialtheoretische}} for $N = 2$. A
modern proof of this result can be found in {\tmname{DiBenedetto}} and
{\tmname{Friedman}} {\cite{dibenedetto1986bubble}} who, in 1985, extended it
to all $N \geqslant 2$. In 1994, {\tmname{Karp}} {\cite{karp1994newtonian}},
by the means of certain topological methods, obtained an alternative proof of
the inverse homogeneous ellipsoid problem.
\end{remark}

Despite the well-knownness of \eqref{eq:quadratic} in the mathematical and physical
community, and its importance in theoretical and applied studies
{\cite{alouges2009magnetization,beleggia2003computation,BrownB1962,DeSimoneA1993,DeSimoneA1995,di2011generalization,OsbornA1945}},
rigorous proofs of that result are not readily available in the literature: to
the best knowledge of the author, relative modern treatments of the interior
problem can be found in {\cite{KelloggB2010}} and
{\cite{stratton2007electromagnetic}}, and more recently in
{\cite{miloh1999note}} and {\cite{pohanka2011gravitational}} where also the
the exterior problem is investigated. However, in all the cited references,
the solution of the problem is always based on the use of ellipsoidal
coordinates which tends to focus the attention on the computational details of
the question rather than on its geometric counterpart. Eventually, modern
proofs of Newton's theorem and relation (\ref{eq:quadratic}), as well as far
reaching beautiful generalizations, can be found in {\cite{khavinson2013tale}}
where neverthless the problem of finding an anlytic expression for the
coeficients $c$ and $P$ is not touched.

Aim of this paper is to give a modern and concise derivation for the
expression of the {\tmem{interior}} and {\tmem{exterior}} Newtonian potential
(induced by homogeneous ellipsoids). The very short argument is essentially
based on the application of {\tmname{Reynolds}} transport theorem in
connection with {\tmname{Green-Stokes}} integral representation formula for
smooth functions on bounded domains of $\RR^N$. This approach permits to
reduce the $N$-dimensional problem to a 1-dimensional one, providing (in
particular) at once a proof of (\ref{eq:quadratic}) together with an explicit
expression of the coefficients $P$ and $c$ in terms of 1-dimensional
integrals. More precisely, the paper is organized as follows: 

Section \ref{sec:one} is devoted to the main result of the paper. We
  give a concise proof of the homogeneous ellipsoid problem. For completensee,
  we then derive {\tmname{Newton}}'s {\tmname{shell theorem}} as an immediate
  corollary. 
	
	In Section\ref{sec:two} we focus the attention to the three
  dimensional case. An expression in terms of the elliptic integrals is given
  for the coefficients of $P$. Due to its physical relevance, particular
  attention is paid to the eigenvalues of $P$. Indeed, when $N = 3$, the
  matrix $P$ and its eigenvalues, known in the theory of ferromagnetism
  respectively as {\tmem{the demagnetization tensor}} and {\tmem{the
  demagnetizing factors}}, are one of the most important and well-studied
  quantities of ferromagnetism
  {\cite{alouges2009magnetization,alouges2013liouville,beleggia2003computation,BrownB1962,di2011generalization,OsbornA1945}}.
  In fact, the following magnetostatic counterpart of the homogeneous
  ellipsoid problem holds: {\tmem{given a uniformly magnetized ellipsoid, the
  induced magnetic field is also uniform inside the ellipsoid}}. This result
  was for the first time showed by {\tmname{Poisson}}
  {\cite{Poisson1825second}}, while an explicit expression for the
  demagnetizing factors was for the first time obtained by {\tmname{Maxwell}}
  {\cite{Maxwell1873}}. Their importance is in that they encapsulate the
  self-interaction of magnetized bodies: their knowledge being equivalent to
  the one of the corresponding demagnetizing (stray) fields
  {\cite{beleggia2003computation}}.
\vspace{20pt}

\section{The Interior and Exterior Potential of an Homogeneous
Ellipsoid}\label{sec:one}
In what follows we denote by $\Omega$ the ellipsoidal domain of $\RR^N$ ($N
\geqslant 3$) having $a_1, a_2, \ldots, a_N$ as semi-axes lengths. We then
denote by $(\Omega_t)_{t \in [0, + \infty)}$ the family of ellipsoidal domains
of $\RR^N$, given by the inverse image $\phi_t^{- 1} (B_N)$ of the unit ball
of $\RR^N$ under the one parameter family of diffeomorphisms $\phi_t : x \in
\RR^N \mapsto \sqrt{A_t} x \in \RR^N$, where
\begin{equation}
  \sqrt{A_t} \assign \mathrm{diag} \left[ \frac{1}{\sqrt{a_1^2 + t}},
  \frac{1}{\sqrt{a_2^2 + t}}, \ldots, \frac{1}{\sqrt{a_N^2 + t}} \right] .
\end{equation}
Note that each diffeomorphism $\phi^{- 1}_t$ maps the unit ball of $\RR^N$
onto the ellipsoidal domain of $\RR^N$ defined by the position $\Omega_t =
\{ x \in \RR^N \; : \; | \phi_t (x) |^2 \leqslant 1\}$. In
particular $\partial \Omega_t =\{ x \in \RR^N \; : \; | \phi_t (x) |^2 =
1 \}$ and $\Omega \equiv \Omega_0$. Finally, we denote by
$\mathcal{N}_{\Omega_t} [1_{\Omega_t}]$ the Newtonian potential generated by
the uniform space density of masses or charges on $\Omega_t$:
\begin{equation}
  \mathcal{N}_{\Omega_t} [1_{\Omega_t}] (x) = c_N \int_{\Omega_t} \frac{1}{| x
  - y |^{N - 2}} \mathd \tau (y),
\end{equation}
with $c_N {\assign} [(N - 2) \omega_N]^{- 1}$ and
$\omega_N$ the surface measure of the unit sphere in $\RR^N$ (cfr.
{\cite{dautray2000mathematical,dibenedetto2010partial}}).

The main result of the paper is stated in the following:

\begin{theorem}
  Let $\Omega =\{ x \in \RR^N \; : \; | \phi_0 (x) |^2 \leqslant 1
  \}$ be the ellipsoidal domain of $\RR^N$ having $(a_1, a_2, \ldots,
  a_N) \in \RR^N_+$ as semi-axes lengths. For every $x \in \RR^N$
  \begin{equation}
    \mathcal{N}_{\Omega} [1_{\Omega}] (x) = \frac{1}{4}  \int_{\tau (x)}^{+
    \infty} \gamma_t  (1 - A_t x \cdot x) \mathd t, \qquad \gamma_t \assign
    \prod_{i = 1}^N \frac{a_i}{\sqrt{a_i^2 + t}}, \label{eq:potentialellipsn}
  \end{equation}
  where we have denoted by $\tau$ the non-negative real valued function
  \begin{equation}
    \tau : x \in \RR^N \mapsto \left\{ \begin{array}{ll}
      0 & \text{ if \ } x \in \bar{\Omega}\\
      \tau (x) & \text{ if \ } x \in \Omega^c \cap \partial \Omega_{\tau (x)}
    \end{array} \right. . \label{eq:indicatrixfunction}
  \end{equation}
\end{theorem}

\begin{remark}
  For general ellipsoidal domains the integral in (\ref{eq:potentialellipsn})
  can be evaluated by the means of the theory of elliptic integrals (cfr.
  {\cite{abramowitz1965handbook,prasolov1997elliptic}}). The function $\tau$
  defined by (\ref{eq:indicatrixfunction}) can be computed by solving the
  equation $\partial \Omega_{\tau (x)} = \{ x \in \RR^N \; : \; |
  \phi_{\tau (x)} (x) |^2 = 1 \}$. In particular, when $\Omega$ is a
  spherical region of radius $a$ (centered around the origin), the function
  $\tau$ reduces to the function equal to $| x |^2 - a^2$ if $x \in \RR^N
  \backslash \Omega$, zero otherwise, and the integral in
  (\ref{eq:potentialellipsn}) can be readily computed.
\end{remark}

\begin{proof}
  For every $t \in \RR^+$, the function $| \phi_t |^2 - 1$ is the unique
  solution of the homogeneous {\tmname{Dirichlet}} problem for the
  {\tmname{Poisson}} equation $u = 2 \tmop{tr} (A_t)$ in $\Omega_t$, $u \equiv
  0$ on $\partial \Omega_t$. Thus, according to the {\tmname{Green-Stokes}}
  representation formula (see {\cite{dibenedetto2010partial}}), for every $t
  \in \RR^+$ we have
  \begin{equation}
    (| \phi_t |^2 - 1) 1_{\Omega_t} = - 2 \tmop{tr} (A_t)
    \mathcal{N}_{\Omega_t} [1_{\Omega_t}] +\mathcal{S}_{\partial \Omega_t}
    [\partial_{\tmmathbf{n}} | \phi_t |^2] \quad \text{in } \RR^N \backslash
    \partial \Omega_t, \label{eq:repgreen}
  \end{equation}
  where we have denoted by
  \begin{equation}
    \mathcal{S}_{\partial \Omega_t} [\partial_{\tmmathbf{n}} | \phi_t |^2] (x)
    \assign c_N \int_{\partial \Omega_t} \frac{\partial_{\tmmathbf{n} (y)} |
    \phi_t (y) |^2}{| x - y |^{N - 2}} \mathd \sigma (y),
  \end{equation}
  the simple layer potential generated by the space density of masses or
  charges $\partial_{\tmmathbf{n}} | \phi_t |^2$ concentrated on $\partial
  \Omega_t$ (cfr. {\cite{dautray2000mathematical,dibenedetto2010partial}}).
  Next, we observe that the $N$-dimensional Newtonian kernel $| x |^{2 - N}$
  is in $W^{1, 1}_{\tmop{loc}} ( \RR^N )$, therefore due to
  {\tmname{Reynolds}} transport theorem,
  \begin{equation}
    \partial_t \mathcal{N}_{\Omega_t} [1_{\Omega_t}] = c_N \int_{\partial
    \Omega_t} \frac{\tmmathbf{v}_t (y) \cdot \tmmathbf{n} (y)}{| x - y |^{N -
    2}} \mathd \sigma (y) = \frac{1}{4} \mathcal{S}_{\partial \Omega_t}
    [\partial_{\tmmathbf{n}} | \phi_t |^2], \label{eq:reynoldstransport}
  \end{equation}
  where we have denoted by $\tmmathbf{v}_t \assign \partial_t [\phi_t^{- 1}]
  \circ \phi_t$ the eulerian velocity field associated to the motion $\phi^{-
  1}_t$, for which one has
  \begin{equation}
    \tmmathbf{v}_t (y) \assign \partial_t [\phi_t^{- 1}] \circ \phi_t (y) =
    \frac{1}{2} A_t y = \frac{1}{4} \nabla | \phi_t (y) |^2 \quad \forall (t,
    y) \in \RR^+ \times \partial \Omega_t . \quad
  \end{equation}
  Hence, substituting (\ref{eq:reynoldstransport}) into (\ref{eq:repgreen}) we
  get
  \begin{equation}
    \frac{1}{4} (| \phi_t (x) |^2 - 1) 1_{\Omega_t} (x) = - \frac{1}{2}
    \tmop{tr} (A_t) \mathcal{N}_{\Omega_t} [1_{\Omega_t}] (x) + \partial_t
    \mathcal{N}_{\Omega_t} [1_{\Omega_t}] (x) \quad \forall x \in \RR^N
    \backslash \partial \Omega_t . \label{eq:greenrepresent3}
  \end{equation}
  Moreover, by the assignment
  \begin{equation}
    \gamma_t \assign \exp \left( - \frac{1}{2} \int_0^t \tmop{tr} (A_s) \mathd
    s \right) = \prod_{i = 1}^N \frac{a_i}{\sqrt{a_i^2 + t}} \quad, \quad
    \gamma_0 = 1,
  \end{equation}
  the equality (\ref{eq:greenrepresent3}) reads as
  \begin{equation}
    \partial_t (\gamma_t \mathcal{N}_{\Omega_t} [1_{\Omega_t}]) = \frac{1}{4}
    \gamma_t (| \phi_t |^2 - 1) 1_{\Omega_t} \quad \forall x \in \RR^N
    \backslash \partial \Omega_t .
  \end{equation}
  Thus, once introduced the non-negative real function defined by
  (\ref{eq:indicatrixfunction}) we have, for every $t \in \RR^+$ and for every
  $x \in \RR^N$
  \begin{equation}
    \partial_t (\gamma_t \mathcal{N}_{\Omega_t} [1_{\Omega_t}]) (x) =
    \frac{1}{4} \gamma_t  (| \phi_t (x) |^2 - 1) 1_{[\tau (x), + \infty)} (t)
    . \label{eq:repgreen2}
  \end{equation}
  Integrating both members of (\ref{eq:repgreen2}) on $[0, + \infty)$; taking
  into account that due to the well-known decay at infinity of the Newtonian
  potential {\cite{SalsaB2010}} one has $\lim_{t \rightarrow + \infty}
  \gamma_t \mathcal{N}_{\Omega_t} [1_{\Omega_t}] = 0$; we finish with
  (\ref{eq:potentialellipsn}).
\end{proof}

\begin{corollary}
  $(\nobracket${\tmname{Newton's shell theorem}}$\nobracket)$ Let $\Omega
  \subseteq \RR^3$ be an homogeneous spherical region (centered around the
  origin) of radius $a$ and of total mass $M$. For every $x \in \RR^3
  \backslash \Omega$ the induced gravitational potential is the same as though
  all of its mass were concentrated at a point at its center. Moreover, for
  all $t > 1$, $\mathcal{N}_{t \Omega \backslash \Omega} [1_{t \Omega
  \backslash \Omega}]$ is constant in $\Omega$, i.e{\tmem{.}} the hollow ball
  induces no gravitational force inside $\Omega$.
\end{corollary}

\begin{proof}
  We denote by $\rho 1_{\Omega} (x)$, ($\rho \assign \nobracket M / | \Omega
  |)$ the uniform density of mass in $\Omega$. The gravitational potential
  induced by $\rho$ in $\RR^3$ is then given by $u_{\rho} = 4 \pi
  G\mathcal{N}_{\Omega} [\rho] = 4 \pi G\mathcal{N}_{\Omega} [1_{\Omega}]$,
  where we have denoted by $G$ the {\tmem{gravitational constant}}. In this
  geometrical setting the function $\tau$ defined by
  (\ref{eq:indicatrixfunction}) reduces to the function $\tau (x) \assign (| x
  |^2 - a^2) 1_{\RR^N \backslash \Omega} (x)_{}$ and the integral in
  (\ref{eq:potentialellipsn}) immediately gives
  \begin{equation}
    u_{\rho} (x) = \frac{2}{3} \pi G \rho \left( \frac{a^2}{2} - | x |^2 
    \right) \quad \text{if } x \in \bar{\Omega} \quad, \quad u_{\rho} (x) =
    \frac{4}{3} \pi a^3 \frac{G \rho}{| x |} = \frac{G_{} M}{| x |} \quad
    \text{if } x \in \RR^3 \backslash \Omega .
  \end{equation}
  Thus, for every $x \in \RR^3 \backslash \Omega$ the induced gravitational
  potential is equal to the one induced by a {\tmname{Dirac}} mass
  concentrated in the center of $\Omega$. The gravitational field is given by
  $\tmmathbf{g} \assign - \nabla u_{\rho}$ and the fact that the hollow ball
  $t \Omega \backslash \Omega$, $t > 1$, induces no gravitational force inside
  $\Omega$ can be immediately seen by splitting the uniform density of mass in
  $t \Omega \backslash \Omega$ in the form $\rho = (M / | t \Omega \backslash
  \Omega |) 1_{t \Omega} - (M / | t \Omega \backslash \Omega |) 1_{\Omega}$.
  Indeed by linearity we get that $u_{\rho}$ is constant in $\Omega$ and
  therefore $\tmmathbf{g}=\tmmathbf{0}$.
\end{proof}

\section{The demagnetizing factors of the general ellipsoid}\label{sec:two}

We know focus on the three-dimensional framework ($N = 3$) and, in particular,
on the so-called demagnetizing factors of the general ellipsoid
{\cite{OsbornA1945}}. To this end we recall that the demagnetizing (stray)
field $\tmmathbf{h} [\tmmathbf{m}]$ associated to a magnetization
$\tmmathbf{m} \in C^{\infty} (\bar{\Omega})$ can be expressed as the gradient
field of a suitable magnetostatic potential $\varphi_{\tmmathbf{m}}$ (see
{\cite{FriedmanA1980,praetorius2004analysis}}). Precisely
$\varphi_{\tmmathbf{m}} \assign - \tmop{div} \mathcal{N}_{\Omega}
[\tmmathbf{m}]$ and $\tmmathbf{h} [\tmmathbf{m}] \assign - \nabla
\varphi_{\tmmathbf{m}}$ in $\RR^3$. In particular, if $\tmmathbf{m}$ is
constant in $\Omega$, then $\varphi_{\tmmathbf{m}} = -\tmmathbf{m} \cdot
\nabla \mathcal{N}_{\Omega} [1_{\Omega}]$. Thus from
$(\ref{eq:potentialellipsn})$:
\begin{equation}
  \varphi_{\tmmathbf{m}} (x) = P x \cdot \tmmathbf{m} \quad, \quad
  \tmmathbf{h} [\tmmathbf{m}] = - P_{} \tmmathbf{m} \quad \forall x \in
  \Omega,
\end{equation}
where we have denote by $P \assign \nabla \mathcal{N}_{\Omega} [1_{\Omega}]$
the diagonal matrix, known in literature as the {\tmem{demagnetizing tensor}},
whose diagonal $i$-entry (the $i$-th {\tmem{demagnetizing factor}}), by virtue
of $(\ref{eq:potentialellipsn})$, is given by
\begin{equation}
  P_i {\assign} \frac{1}{2} \int_0^{+ \infty}
  \frac{1}{(a_i^2 + t)}  \prod_{j = 1}^3 \frac{a_j}{\sqrt{a_j^2 + t}} \mathd t
  \qquad \forall i \in \NN_3 . \label{eq:demagfactors}
\end{equation}
\begin{proposition}
  We have $P_i \geqslant 0$ for every $i \in \NN_3$ and if $a_1 \geqslant a_2
  \geqslant a_3$ then $P_1 \leqslant P_2 \leqslant P_3$. The trace of $P$
  satisfy the relation $\tmop{tr} (P) = 1$.
\end{proposition}

\begin{proof}
  The first statement is obvious. The relation $\tmop{tr} (P) = 1$ can of
  course be verified by a direct evaluation of the integrals in
  (\ref{eq:demagfactors}), but it is also possible to observe that since the
  Newtonian potential $\mathcal{N}_{\Omega} [1_{\Omega}]$ satisfies the
  {\tmname{Poisson}} equation $\Delta \mathcal{N}_{\Omega} [1_{\Omega}] = -
  1_{\Omega}$, one has $\tmop{tr} (P) = \tmop{div} (P_{} x) = - \Delta
  \mathcal{N}_{\Omega} [1_{\Omega}] = 1_{\Omega}$.
\end{proof}

Assuming $a_1 \geqslant a_2 \geqslant a_3$, from (\ref{eq:demagfactors}) and
the theory of elliptic integrals we get
\begin{eqnarray}
  P_1 & = & 1 - P_2 - P_3 \\
  P_2 & = & - \frac{a_3}{a_2^2 - a_3^2} \left[ a_3 - \frac{a_1 a_2}{(a_1^2 -
  a_2^2)^{1 / 2}} E \left( \arccos \left( \frac{a_2}{a_1} \right) \left|
  \frac{a_1^2 - a_3^2}{a_1^2 - a_2^2} \right. \right) \right] \\
  P_3 & = & + \frac{a_2}{a_2^2 - a_3^2} \left[ a_2 - \frac{a_1 a_3}{(a_2^2 -
  a_3^2)^{1 / 2}} E \left( \arccos \left( \frac{a_3}{a_1} \right) \left|
  \frac{a_1^2 - a_2^2}{a_1^2 - a_3^2} \right. \right) \right], 
\end{eqnarray}
where, for every $y \in \RR$ and every $0 < p < 1$ we have denoted by
\begin{equation}
  E (y | p \nobracket) \assign \int_0^y (1 - p \sin^2 \theta)^{1 / 2} \mathd
  \theta,
\end{equation}
the incomplete elliptic integral of the second kind expressed in
{\tmem{parameter form}} {\cite{abramowitz1965handbook,prasolov1997elliptic}}.
In particular, in the case of a {\tmem{prolate spheroid}} ($a_1 \geqslant a_2
= a_3$) we get
\begin{eqnarray}
 P_1 &=& - \frac{a_3^2}{(a_1^2 - a_3^2)^{3 / 2}} \left[ (a_1^2 - a_3^2)^{1 / 2}
  + a_1 \tmop{arccoth} \left( \frac{a_1}{(a_1^2 - a_3^2)^{1 / 2}} \right)
  \right] ,\\ 
P_2 &=& P_3 = \frac{1 - P_1}{2},
\end{eqnarray}
while in the case of an {\tmem{oblate spheroid}} ($a_1 = a_2 \geqslant a_3$)
\begin{eqnarray}
  P_1 &=& P_2 = \frac{1 - P_3}{2},\\ 
P_3 &=& \frac{a_1^2}{(a_1^2 -
  a_3^2)^{3 / 2}} \left[ (a_1^2 - a_3^2)^{1 / 2} + a_3 \arctan \left(
  \frac{a_3}{(a_1^2 - a_3^2)^{1 / 2}} \right) - a_3 \frac{\pi}{2} \right] .
\end{eqnarray}
Finally, in the case of a {\tmem{sphere}} ($a_1 = a_2 = a_3$) one finish with
$P_1 = P_2 = P_3 = \frac{1}{3}$.

\section{Acknowledgments}
This work was partially supported by the labex LMH through the grant ANR-11-LABX-0056-LMH in the ``\emph{Programme des Investissements d'Avenir}''.

%
%
%
%
%


\end{document}